\documentclass[11pt]{amsart}
\usepackage{amsfonts, amsmath, amssymb, color, soul}
\usepackage{etex}   
\usepackage{amsthm}
\usepackage{booktabs}
\usepackage{hhline}
\usepackage{lipsum}
\usepackage{lscape}
\usepackage{subfig}
\usepackage{textcomp}
\usepackage{tikz}
\usetikzlibrary{decorations.pathmorphing,shapes,arrows,positioning}
\usepackage[all]{xy}
\usepackage{xcolor}
\usepackage{young}
\usepackage{youngtab}
\usepackage{ytableau}
\usepackage[colorlinks=true, linkcolor=blue, citecolor=magenta, urlcolor=blue, backref=page]{hyperref}   
\theoremstyle{plain}
\newtheorem{thm}{Theorem}[section]

\newtheorem{prop}[thm]{Proposition}
\newtheorem{cor}[thm]{Corollary}
\newtheorem{rem}[thm]{Remark}

\theoremstyle{definition}

\newcommand{\la}{\lambda}

\newcommand{\bl}{\boldsymbol{\lambda}}
\newcommand{\bmu}{\boldsymbol{\mu}}

\newcommand{\tbtau}{\tilde{\boldsymbol{\tau}}}

\newcommand{\tbl}{\tilde{\boldsymbol{\lambda}}}
\newcommand{\trh}{\tilde{\boldsymbol{\rho}}}
\newcommand{\f}{\mathbb{F}}

\numberwithin{equation}{section} \errorcontextlines=0

\newcommand{\GL}{\mathrm{GL}}
\newcommand{\X}{\mathcal{X}}

\begin{document}
\title{On character values of $\GL_n(\mathbb F_q)$}
\author{Naihuan Jing}
\address{Department of Mathematics, North Carolina State University, Raleigh, NC 27695, USA}
\email{jing@ncsu.edu}
\author{Yu Wu}
\address{School of Mathematics, South China University of Technology,
Guangzhou, Guangdong 510640, China}
\email{wywymath@163.com}
\thanks{Partially supported by
Simons Foundation grant  MP-TSM-00002518 and
NSFC grant 12171303.}

\begin{abstract}
In this paper, we use vertex operator techniques to compute character values on unipotent classes of $\GL_n(\mathbb F_q)$. By realizing the Grothendieck ring $R_G=\bigoplus_{n\geq0}^\infty R(\GL_n(\mathbb F_q))$ as
Fock spaces, we formulate the Murnaghan-Nakayama rule of $\GL_n(\mathbb F_q)$ between Schur functions colored by an orbit $\phi$ of linear characters of $\overline{\mathbb F}_q$ under the Frobenius automorphism on and modified Hall-Littlewood functions colored by $f_1=t-1$, which provides detailed information on
the character table of $\GL_n(\mathbb F_q)$. As applications, we use vertex algebraic methods to determine the Steinberg characters of  $\GL_n(\mathbb F_q)$, which were previously determined by Curtis-Lehrer-Tits via geometry of homology groups of spherical buildings and Springer-Zelevinsky utilizing Hopf algebras.
\end{abstract}
\maketitle

\section{Introduction}
 Frobenius introduced the concept of characters of finite groups by extending those of commutative groups in number theory. One of the first examples is determination of all irreducible characters of the symmetric groups using the Frobenius characteristic map and symmetric functions. Schur successively generalized the Frobenius characteristic map for the double covering groups of the symmetric group aka the spin groups \cite{Sc}. Specht obtained Frobenius-type character formula for the wreath products \cite{Sp}.  In a renowned paper \cite{jaG}, Green identified all irreducible characters of $\GL_n(\f_q)$ and further developed the correspondence between symmetric functions and representation theory. In \cite{FJW} we reformulated Specht's theory as applications of the McKay correspondence and vertex representations of affine Lie algebras, which had generalized the first author's earlier work \cite{Jing1} on spin groups.

 The classification of irreducible representations of finite Chevalley groups can be reduced to that of unipotent representations of the endoscopic groups, which leads to a program to essentially determine all irreducible characters \cite{DL, L, L1, L2} (see also \cite{DM, GM}).

 Recently
we have studied character theory of $\GL_n(\f_q)$ by realizing the Grothendieck ring $R_G$ as two interrelated Fock spaces  \cite{JW} and expressed the irreducible character values $\chi^{\tbl}_{\bmu}$
 as matrix coefficients of two special vertex operators, where the partitions $\tbl$ and $\bmu$ are colored by
 orbits of the Frobenius automorphisms. The goal of this paper is to
 apply techniques of vertex operators \cite{Jing1, Jing2, JW} to compute characters at unipotent classes.

The unipotent conjugacy classes of $\GL_n(\f_q)$ are in one-to-one correspondence to
partition-valued functions colored only by the monic polynomial $t-1$. L\'ubeck and Malle \cite{LM} have formulated a Murnaghan-Nakayama rule for general Chevalley groups using explicit information of Lusztig restriction.
We also obtain a direct Murnaghan-Nakayama-like rule for $\GL_n(\f_q)$ on the unipotent elements by vertex operators, bypassing reference to Green polynomials \cite{Mor, JL}. The rule determines character values on the unipotent classes as explicit linear combinations of the irreducible characters on unipotent classes in lower rank groups.

As applications, we also use vertex algebraic techniques to compute the Steinberg characters in Prop. \ref{St}, which were previously determined by Curtis, Lehrer and Tits \cite{CLT} using rational homology groups and by Springer and Zelevinsky \cite{SZ} using Hopf algebras \cite{Ze}.

The structure of this article is as follows. In Sect.2 after reviewing partitions and basic notions, we recall two Heisenberg algebras in \cite{JW} which provide Fock space realization of the Gorthendieck group. We express the character values on unipotent conjugacy classes as matrix coefficients of vertex operators in \eqref{ful:char}. In Sect.3, we introduce a family of symmetric functions $q_r(X_{f_1};q)$ to present the base change formula in Theorem \ref{thm:q-S} for two important bases $\{p_n(f)\}$ and $\{p_n(\phi)\}$ of $\mathcal F_G$. We then show the Murnaghan-Nakayama rule (Theorem \ref{MN}, Corollary \ref{MMN} and Theorem \ref{table}) and a new recursive formula \eqref{cha table}. In the last part of the section, We
study the characters $\chi^{(m)(\phi)}_{\mu(f_1)}$ and $\chi^{\la(\phi)}_{(n)(f_1)}$ in Proposition \ref{single1}, we also prove the famous Steinberg character in Proposition \ref{St}. Sect.3 also deals with the case of $\mu(f_1)$ being hook-type (Proposition \ref{hook}).

\section{Preliminaries}

A {\it partition} $\la=(\la_1,\la_2,\ldots)$ of weight $n$, denoted as $\la\vdash n$, is a weakly decreasing sequence of nonnegative integers such that $\sum_i\la_i=|\la|=n$. The number $l(\la)$ of nonzero parts $\la_i$ is called the {\it length}. When the parts are not necessarily descending, $\la$ is called a {\it composition} of $n$, denoted as $\la\vDash n$.
The set of all partitions (resp. compositions) will be denoted as $\mathcal P$ (resp. $\mathcal C$). Two compositions $\la\subset\mu$ if $\la_i\leq\mu_i$ for any $i\geq1$. For two composition $\la=(\la_1,\ldots,\la_l)$ and $\mu=(\mu_1,\ldots,\mu_k)$, we define $[\la,\mu]=(\la_1,\ldots,\la_l,\mu_1,\ldots,\mu_k)$. For example, $[(231),(22)]=(23122)$.

Let $\overline{\mathbb F}_q$ be the algebraic closure of the finite field $\mathbb F_q$ with the Frobenius automorphism $F$: $x\mapsto x^q$. Let $\mathbb F_{q^n}$ be the canonical field extension of $\mathbb F_q$ of degree $n$. Let $M$ denote the multiplicative group of $\overline{\mathbb F}_q$, and let $M_n$ be the multiplicative group of $\mathbb F_{q^n}\subset\overline{\mathbb F}_q$, therefore $M=\bigcup_{n\geq1}M_n$ and $M_m\subset M_n$ if $m$ divides $n$.

Let $M_n^*$ be the group of complex characters of $M_n$ and let $M^*=\bigcup_{n\geq1}M_n^*$ be the character group of $M$, so $F$ induces an automorphism on $M^*$. We have $M_m^*\subset M_n^*$ if $m$ divides $n$.

Let $\Phi$ (resp. $\Phi^*$) be the set of $F$-orbits in $M$ (resp. $M^*$). We also denote by $\Phi_k$ (resp. $\Phi^*_k$) the set of $F^k$-orbits in $M$ (resp. $M^*$). Every orbit in $\Phi$ corresponds to a monic irreducible polynomial over $\mathbb F_q$ excluding $t$. For example, $\{x,x^q,\ldots,x^{q^{d-1}}\}$ is an orbit of $\Phi$, then the corresponding monic polynomial $f=\prod_{i=0}^{d-1}(t-x^{q^i})$.

We denote by $|\phi|$ the number of elements in the orbit $\phi\in\Phi^*$. Let $\X$ be either $\Phi$ or $\Phi^*$, we define a {\it $\X$-colored partition} to be a partition-valued function \cite{IJS} $\X\to\mathcal P$, which is a sequence partitions indexed by $\X$. The {\it weight} of a $\X$-colored partition $\bl$ is defined by
\begin{align}
    \|\bl\|=\sum_{x\in\X}|x||\bl(x)|,\quad\mbox{where}\; |x|=\begin{cases} |x| &\mbox{if $x\in\Phi^*$}\\
    d(x) &\mbox{if $x\in\Phi$.}
\end{cases}
\end{align}
Clearly, $\Phi_n=\{f\in\Phi|\,d(f)|n\}$ and $\Phi_n^*=\{\phi\in\Phi^*|\,d(\phi)|n\}$. We define
\begin{align}
    \mathcal P_n^{\X}&=\{\mbox{$\X$-colored partitions $\bl$ such that $\|\bl\|=n$}\},\\
    C_n^{\X}&=\{\mbox{$\X$-colored compositions $\bl$ such that $\|\bl\|=n$}\}.
\end{align}

For $\bl,\bmu\in\mathcal C_n^{\X}$, we denote $\bl\subset\bmu$ if $\bl(x)\subset\bmu(x)$ for each $x\in\X$. Let
\begin{align}
    n(\bl)&=\sum_{x\in\X}|x|n(\tbl(x)), \quad\mbox{where}\; n(\la)=\sum_{i=1}^{l(\la)}(i-1)\la_i,\\
    Z_{\tbl}&=\prod_{\phi\in\Phi}z_{\tbl(\phi)}, \quad\quad\mbox{where}\; z_{\la}=\prod_{i=1}i^{m_i}m_i!\;\mbox{for $\la=(1^{m_1}2^{m_2}\ldots)$}.
\end{align}

\begin{prop}\cite{jaG,Mac}
The conjugate classes of $G_n=\mathrm{GL}_n(\mathbb F_q)$ are parameterized by $\Phi$-colored partitions $\bmu\in\mathcal P_n^{\Phi}$ $\bmu:\Phi\to\mathcal{P}$, which are called the cycle-types of the conjugacy classes. We denote the set of partition corresponding to unipotent conjugacy classes by $\{\mu(f_1),\mu\vdash\mathcal P\}$. Moreover, the unipotent conjugacy classes of $G_n$ correspond to $\bmu\in\mathcal P_n^{\Phi}$ such that $\bmu(f)=0$ for each $f\neq f_1,f_1= t-1$.

\end{prop}

Let $R(G_n)$ be the Hilbert space of complex-valued class functions on $\GL_n(\mathbb F_q)$ equipped with the canonical scalar product given by $(f, g)=\frac1{|G_n|}\sum_{x\in G_n}f(x)g(x^{-1})$. We introduce the infinite-dimensional graded space
\begin{align}
    R_G=\bigoplus_{n\geq0}R(G_n)
\end{align}
which becomes an associative algebra
under the multiplication $f\circ g=\mathrm{Ind}_{P_{(m,n)}}^{G_{m+n}}f\otimes g$ for
$f\in R(G_m)$ and $g\in R(G_n)$, where $P_{(m,n)}$ is the parabolic subgroup of $G_n$ consisting of
matrices of upper triangular matrices of shape $m\times n$. Then $R_G$ can be
realized as a tensor product of Hall algebras \cite{jaG}\cite{JW}\cite[Chap.IV.3]{Mac}:
$$R_G=\bigotimes_{f\in\Phi}H(\mathbb F_q[t]_{(f)})\otimes_{\mathbb Z}\mathbb C.$$

Let $p_n(\phi),\phi\in\Phi^*$ denote the n-th power sum function of the independent variables $X_{i,\phi}$, $p_n(\phi)=p_n(X_{1,\phi},X_{2,\phi},\ldots)=\sum_i X_{i,\phi}^n$. Let
\begin{align}
    \mathcal F_G=\bigotimes_{\phi\in\Phi^*}\mathbb C[p_n(\phi):n\geq 1]\simeq\bigotimes_{\phi\in\Phi^*}\Lambda_{\mathbb C}
\end{align}
be the polynomial algebra generated by the $p_n(\phi)$, $\mathcal{F}_G$ is a graded algebra with degree given by $\mathrm{deg}(p_n(\phi))=n|\phi|$. Note that $\mathcal F_G$ is the unique level one irreducible representation of the infinite-dimensional Heisenberg Lie algebra $\widehat{\mathfrak h}_{\hat{\overline{\mathbb F}}_q}$ spanned by the $p_n(\phi)$ and the central element $C$ subject to the relation:
\begin{align}\label{e:Heisenberg1}
[p_m(\phi), p_n(\phi')]=m\delta_{m, -n}\delta_{\phi, \phi'}C,
\end{align}
which equips this Fock space with the inner product $\langle p_{\lambda}(\phi), p_{\mu}(\phi')\rangle=z_{\lambda}\delta_{\la\mu}\delta_{\phi, \phi'}$. More generally,
\begin{align}
    \langle p_{\bl},p_{\bmu} \rangle=Z_{\bl}\delta_{\bl\bmu},\quad\mbox{where}\; p_{\bl}=\prod_{\phi\in\Phi}p_{\tbl(\phi)}(\phi).
\end{align}
For each $\phi\in\Phi^*$, the action $p_n(\phi): \mathcal{F}_G\to\mathcal{F}_G$ is regarded as multiplicative operator of degree $n|\phi|$, and the adjoint operator is the differential operator $p_n^*(\phi)=n\dfrac{\partial}{\partial p_n(\phi)}$ of degree $-n|\phi|$.

We introduce the Schur vertex operator $S(\phi;z)$ as the following map $\mathcal{F}_G\to\mathcal{F}_G[[z,z^{-1}]]$:
\begin{align}\label{O:S}
S(\phi; z)&=\mbox{exp} \left( \sum\limits_{n\geqslant 1} \dfrac{1}{n}p_n(\phi)z^{n} \right) \mbox{exp} \left( -\sum \limits_{n\geqslant 1} \frac{\partial}{\partial p_n(\phi)}z^{-n} \right)=\sum_{n\in\mathbb Z}S_n(\phi)z^{n}.
\end{align}
The operators $S_n(\phi)\in\mathrm{End}(\mathcal F_G)$ are the vertex operators realizing the Schur functions,
it is known from \cite{Jing1,Jing3} that $S_{\tbl(\phi)}.1=S_{\tbl(\phi)}\cdots S_{\tbl(\phi)_l}.1=s_{\tbl(\phi)}(\phi)$ in the power sum $p_{\tbl(\phi)}(\phi)$, and $S_{-n}(\phi).1=\delta_{n,0}$ for $n\geq 0$.

For each monic irreducible polynomial $f\in\Phi$, introduce the power-sum symmetric functions in the second set of variables $Y_{i,f}$ via the transform
\begin{align}\label{linear1}
    p_n(f)=\dfrac{(-1)^{n|f|-1}}{q^{n|f|}-1}\sum_{\xi\in M^*_{n|f|}}\overline{\xi(x)}p_{n|f|/|\phi|}(\phi_\xi),
\end{align}
where $x$ is a fixed root of $f$, $\phi_\xi\in\Phi^*$ is the orbit of $\xi$ and $p_r(\phi)=0$ unless $r$ is a nonnegative integer. We remark that any fixed root $x$ of $f$ gives the same $p_n(f)$ in \eqref{linear1}. By the orthogonality of irreducible characters of the group $M_n$,
\begin{align}\label{linear2}
    p_n(\phi)=(-1)^{n|\phi|-1}\sum_{x\in M_{n|\phi|}}\xi(x)p_{n|\phi|/|f|}(f_x)
\end{align}
where $\xi$ is a fixed character $\in\phi$, $f_x$ is the orbit generated by $x$, $p_r(f)=0$ unless $r$ is a nonnegative integer. In particular $f=f_1=t-1$,
\begin{align}\label{linear2:f=1}
    p_n(f_1)=\dfrac{(-1)^{n-1}}{q^n-1}\sum_{\phi\in\Phi_n^*}|\phi|p_{n/|\phi|}(\phi).
\end{align}

\begin{prop}\cite[Prop. 3.2]{JW}
    For any two polynomials $f_1,f_2\in\Phi$, we have
    \begin{align}
        \langle p_n(f_1),p_m(f_2) \rangle=\dfrac{n}{q^{n|f|}-1}\delta_{n,m}\delta_{f_1,f_2}
    \end{align}
Moreover, for any two partitions $\la,\mu$
    \begin{align}
         \langle p_{\la}(f),p_{\mu}(f')\rangle=z_{\la}(q^{-|f|})q^{-|\la||f|}\delta_{\la\mu}\delta_{f,f'}
    \end{align}
\end{prop}


For each $\phi\in\Phi^*$ since $p_n(\phi)$ is a linear combination of $p_n(f),f\in\Phi$, $\mathcal F_G$ can be regarded as the space
\begin{align}\label{e:FG}
\mathcal{F}_G=\bigotimes_{f\in \Phi}\mathbb{C}[p_n(f):n\geqslant1]=\bigotimes_{f\in \Phi}\Lambda_{\mathbb C}(f),
\end{align}
where $\Lambda_{\mathbb C}(f)=\mathbb C[p_1(f), p_2(f), \ldots]\cong \Lambda_{\mathbb C}$ is the space of symmetric functions generated by the $p_n(f)$.

The Fock space $\mathcal F_G$ is also the unique level one irreducible representation of the infinite dimensional Heisenberg Lie algebra $\widehat{\mathfrak{h}}_{\overline{\mathbb F}_q}$ generated by the $p_n(f)$ and the central element $C$ with
the relation
 \begin{align}\label{e:Heisenberg2}
[p_m(f), p_n(f')]=\frac{m}{q^{m|f|}-1}\delta_{m, -n}\delta_{f, f'}C. \qquad\quad (m>0)
\end{align}
For each $f\in\Phi$, the action $p_n(f):\mathcal{F}_G\to\mathcal{F}_G$ is the multiplicative operator of degree $n|f|$, and the adjoint operator is the differential operator $p_n^*(f)=\dfrac{n}{q^{n|f|}-1}\dfrac{\partial}{\partial p_n(f)}$ of degree $-n|f|$. Note that $*$ is $\mathbb C$-linear and anti-involutive, for $u,v\in\mathcal{F}_G$, we have
\begin{align}
  \langle p_n(f)u,v\rangle=\langle u,p_n^*(f)v \rangle.
\end{align}

For each $f\in\Phi$, we define the vertex operator $Q(f;z)$ and $Q^*(f;z)$ as the following product of exponential operators: $\mathcal F_G\to\mathcal F_G[[z,z^{-1}]]$\cite{Jing2,JW}:
\begin{align}\label{O:Q}
    Q(f; z)&=\mbox{exp} \left( \sum\limits_{n\geq 1} \dfrac{q^{n|f|}-1}{n}p_n(f)z^{n} \right) \mbox{exp} \left( -\sum \limits_{n\geq 1} q^{-n|f|}\frac{\partial}{\partial p_n(f)}z^{-n} \right)=\sum_{n\in\mathbb Z}Q_n(f)z^{n},\\\label{O:dual Q}
    Q^*(f;z)&=\mbox{exp}\left( -\sum\limits_{n\geq 1} \dfrac{1-q^{-n|f|}}{n}p_n(f)z^{n} \right) \mbox{exp} \left( \sum \limits_{n\geq 1} \frac{\partial}{\partial p_n(f)}z^{-n} \right)=\sum_{n\in\mathbb Z}Q_n^*(f)z^{-n}.
\end{align}
The operators $Q_n(f),Q_n^*(f)\in\mathrm{End}(\mathcal F_G)$ are vertex operators realizing Hall-Littlewood functions defined first in \cite{Jing2}, where it was shown that $q^{-|\la||f|}Q_{\bmu(f)}.1=q^{-|\la||f|}Q_{\bmu(f)_1}\ldots Q_{\bmu(f)_l}.1=Q_{\bmu(f)}(f)$ in the $p_n(f)$, and $Q_{-n}(f).1=\delta_{n,0}$, $Q_n^*(f).1=\delta_{n,0}$ for $n\geq0$.


\begin{prop}\cite{JW} The irreducible characters of $G_n$ are parameterized by $\Phi^*$-colored partitions $\tbl\in\mathcal P_n^{\Phi^*}$.The value of the irreducible character $\chi^{\tbl}$ at an unipotent conjugacy class of type $\mu(f_1)$ is simplified to
\begin{align}\label{ful:char}
\chi^{\tilde{\bl}}_{\mu}&=q^{n(\mu)}\left\langle \prod_{\phi\in\Phi^*}S_{\tbl(\phi)}.1, Q_{\mu(f_1)}(f_1).1\right\rangle.
\end{align}
\end{prop}

\section{The Murnaghan-Nakayama rule }


We introduce  a new family of symmetric functions $q_n(Y_f;q^{|f|}),f\in\Phi$ defined by 
\begin{equation}\label{O:q}
\sum_{r\geq0}q_r(Y_f;q^{|f|})z^n=\exp\left({-\sum_{n\geq1}\dfrac{1-q^{-n|f|}}{n}p_n(f)z^n}\right)
\end{equation}
It follows from the action of \eqref{O:dual Q}
that for $n\geq 0$
\begin{align}\label{q}
   q_n(Y_f;q^{|f|})=Q_{-n}^*(f).1=q^{-n|f|}\sum_{\la\vdash n}\dfrac{p_{\la}(f)}{z_{\la}(q^{|f|})}.
\end{align}
Thus $Q^*_{-n}(f).1$ is just the Hall-Littlewood function associated to one row partition $(n)$ up to a factor. 

\begin{thm}\label{thm:q-S}
The function $q_n(Y_{f_1};q)$ can be expressed in $p_{\tbl}=\prod\limits_{\phi\in\Phi^*}p_{\tbl(\phi)}(\phi)$ as follows.
 \begin{align}
        q_n(Y_{f_1};q)=\left(\frac{-1}{q}\right)^{n}\sum_{\tbl}\dfrac{1}{Z_{\tbl}}p_{\tbl}
\end{align}
    summed over all $\tbl\in\mathcal P_n^{\Phi^*}$. Moreover, $q_n(Y_{f_1},q)$ is expanded in $\prod\limits_{\phi\in\Phi}s_{\tbl(\phi)}(\phi)$ by
\begin{align}
        q_n(Y_{f_1};q)=\left(\frac{-1}{q}\right)^{n}\sum_{\tbl}\prod_{\phi\in\Phi^*}s_{\tbl(\phi)}(\phi)
\end{align}
   summed over all $\tbl\in\mathcal P_n^{\Phi^*}$ with $l(\tbl(\phi))=1$ for each $\Phi\in\Phi^*$.
\end{thm}
\begin{proof}
    Notice that $p_{n/d(\phi)}=0$ if $d(\phi)$ is not a divisor of $n$. By definitions \eqref{O:dual Q}\eqref{q} and relation \eqref{linear2:f=1}, we have
    \begin{align*}
\sum_{n\in\mathbb Z}(Q_{-n}^*(f_1).1)z^{-n}
&=\exp\left(\sum_{n\geq1}\dfrac{(-1)^n}{nq^n}\sum_{\phi\in\Phi^*_n}|\phi|p_{n/|\phi|}(\phi) z^n\right)\\
&=\prod_{\phi\in\Phi^*}\exp\left(\sum_{n\geq1}\dfrac{(-1)^{n|\phi|}}{nq^{n|\phi|}}p_n(\phi)z^{n|\phi|} \right)\\
&=\prod_{\phi\in\Phi^*}\sum_{\la}\dfrac{1}{z_\lambda}p_{\la}(\phi)\left(\dfrac{z}{-q}\right)^{|\la||\phi|}.
    \end{align*}
which proves the first statement. For the second equation, it is clear by using the relation
    \begin{align}
        Q^*(f_1;z).1=\prod_{\phi\in\Phi^*}S(\phi;(-z/q)^{|\phi|}).1
    \end{align}
Then taking the coefficient of $z^n$ of both sides gives the result.
\end{proof}

For a given $\phi\in\Phi^*$  we simply write $S_n$ for $S_n(\phi)$ and $Q_n$ for $Q_n(f_1)$ unless otherwise noted. We first recall the normal ordering operator: 
\begin{align*}
:Q^*(f;z)S(\phi;w):&=\mbox{exp}\left(\sum\limits_{n\geq 1} \dfrac{q^{-n|f|}-1}{n}p_n(f)z^{n} +\dfrac{1}{n}p_n(\phi)w^{n}\right)\mbox{exp}\left( \sum \limits_{n\geq 1} \frac{\partial}{\partial p_n(f)}z^{-n}-\frac{\partial}{\partial p_n(\phi)}w^{-n}\right).
\end{align*}

\begin{prop}\label{rel:S-H}
    The commutation relations between the vertex operators realizing Hall-Littlewood functions colored by $f_1$ and Schur functions colored by $\phi\in\Phi^*$ are
    \begin{align}\label{relation: dualQ-S}
    Q_n^*S_m=\sum_{i=1}^\infty A(q^{|\phi|};i)(-1)^{i|\phi|}S_{m-i}Q_{n-i|\phi|}^*
    \end{align}
where $A(q^{|\phi|};i)=\begin{cases}
    q^{i|\phi|}(q^{-|\phi|}-1),&i>1,\\
    -q^{|\phi|}.&i=1.
\end{cases}$
\end{prop}
\begin{proof}
By relations \eqref{linear2}\eqref{e:Heisenberg2}, we compute the brackets by
\begin{align*}
 &\exp\left\{\left[\sum_{n\geq1}\frac{\partial}{\partial p_n(f_1)}z^{-n},\sum_{n\geq1}\frac{1}{n}p_n(\phi)w^n\right] \right\}=\exp\left\{\sum_{n\geq1}-\frac{1}{n}\left(\dfrac{w}{(-z)^{|\phi|}}\right)^n  \right\}
=1-\dfrac{w}{(-z)^{|\phi|}}.
\end{align*}
    Using usual technique of vertex operators, we have
\begin{align*}
    Q^*(f_1;z)S(\phi;w)&=:Q^*(f_1;z)S(\phi;w):\left(1-\dfrac{w}{(-z)^{|\phi|}}\right)\\
    S(\phi;w)Q^*(f_1;z)&=:Q^*(f_1;z)S(\phi;w):\left(1-\dfrac{(-z)^{|\phi|}}{wq^{|\phi|}}\right)
\end{align*}Hence we obtain the equation
\begin{align}\label{Q*-S}
    Q^*(f_1;z)S(\phi;w)\left( 1-\dfrac{(-z)^{|\phi|}}{wq^{|\phi|}}\right)&=S(\phi;w)Q^*(f_1;z)\left(1-\dfrac{w}{(-z)^{|\phi|}}\right).
\end{align}
Note that the polynomial
\begin{align*}
(1-\dfrac{w}{(-z)^{|\phi|}})(1-\dfrac{(-z)^{|\phi|}}{wq^{|\phi|}})^{-1}=\left(\dfrac{w}{(-z)^{|\phi|}}-1\right)\sum_{i=1}^{\infty}\dfrac{w^iq^{i|\phi|}}{(-z)^{i|\phi|}}.
\end{align*}
Then taking the coefficient of $z^{-n}w^m$ of both sides gives the result.
\end{proof}

Following the \cite{Jing1,Jing3}, it is known that for integers $m,n$
\begin{align}\label{rel:Schur}
        S_m(\phi)S_n(\phi)+S_{n-1}(\phi)S_{m+1}(\phi)=0.
    \end{align}
For any two $\phi,\psi\in\Phi^*$, the operators $S_n(\phi)$ and $S_m(\psi)$ commute with each other.

There appear infinitely many terms in the expansion of $Q_n^*S_m$. Using the relations of $Q_{n-(m+1)|\phi|}^*S_{-1}$ we can convert $Q_n^*S_m$ to a finite sum. Namely,
\begin{align}\label{infinite rel:Q-S}
    Q_n^*S_m=&\sum_{i=1}^{m+1}A(q^{|\phi|};i)(-1)^{i|\phi|}S_{m-i}Q_{n-i|\phi|}^*+(-1)^{(m+2)|\phi|}q^{(m+1)|\phi|}S_{-2}Q_{n-(m+2)|\phi|}^*\\\notag
    &+(-1)^{(m+1)|\phi|}q^{(m+1)|\phi|}Q_{n-(m+1)|\phi|}^*S_{-1}
\end{align}
We remark that $S_{-2}Q_{n-(m+2)|\phi|}^*.1=0$ by the relation \eqref{rel:Schur} and Theorem \ref{thm:q-S}, then
\begin{align}
    Q_n^*S_m.1=\sum_{i=1}^{m+1}A(q^{|\phi|};i)(-1)^{i|\phi|}S_{m-i}Q_{n-i|\phi|}^*.1
\end{align}


For a partition $\la=(\la_1,\ldots,\la_l)$, we define the following partitions associated with $\la$:
\begin{align}
    \la^{[i]}&=(\la_{i+1},\ldots,\la_l),\qquad i=0,1,\ldots,l\\
    \underline{\la}&=\la+\delta+(1^l), \quad\mbox{where}\;\; \delta=(l-1,l-2,\ldots,1,0)
\end{align}

The following gives a key step for the Murnaghan-Nakayama rule for character values.
\begin{thm}\label{MN}
    For a partition $\la=(\la_1,\ldots,\la_l)$ and a positive integer $k$, we have
    \begin{align}\label{finiite rel:Q-S}
Q_k^*S_{\la}.1=\sum_{\tau\subset\underline{\la}}\prod_{i=1}^lA(q^d;\tau_i)(-1)^{|\tau|d}S_{\la-\tau}Q_{k-|\tau|d}^*.1
    \end{align}
summed over all  compositions  $\tau\subset \underline{\la}$ such that  $\tau_i>0$ for $1\leq i\leq l$.
\end{thm}
\begin{proof}
    We argue by induction on $l(\la)$. The initial step is clear. We assume that \eqref{finiite rel:Q-S} holds for any partition with length $<l(\la)$. By relation \eqref{rel:Schur} it follows from the recursion formula \eqref{infinite rel:Q-S} that
    \begin{align*}
    Q_k^*S_{\la}.1=&\sum_{i=1}^{\la_1+l}A(q^{|\phi|};i)(-1)^{i{|\phi|}}S_{\la_1-i}Q^*_{k-i|\phi|}S_{\la^{[1]}}.1+(-1)^{(\la_1+l+1)|\phi|}q^{(\la_1+l)|\phi|}S_{-l-1}Q^*_{k-(\la_1+l+1)|\phi|}S_{\la^{[1]}}.1\\
    &+(-1)^{(\la_1+l)|\phi|}q^{(\la_1+l)|\phi|}Q_{k-(\la_1+l)|\phi|}S_{-l}S_{\la^{[1]}}.1\\
    =&\sum_{i=1}^{\la_1+l}A(q^{|\phi|};i)(-1)^{i|\phi|}S_{\la_1-i}\sum_{\tau\subset\underline{\la^{[1]}}}\prod_{j=1}^{l-1}A(q^{|\phi|};\tau_j)(-1)^{|\tau||\phi|}S_{\la^{[1]}-\tau}Q^*_{k-i|\phi|-|\tau||\phi|}.1\\
=&\sum_{\tau\subset\underline{\la}}\prod_{i=1}^lA(q^{|\phi|};\tau_i)(-1)^{|\tau||\phi|}S_{\la-\tau}Q^*_{k-|\tau||\phi|}.1
    \end{align*}
\end{proof}

Next we consider the action of $Q_k^*$ on $\prod\limits_{\phi\in\Phi^*}S_{\tbl(\phi)}.1$ for $\tbl\in\mathcal{P}^{\Phi^*}$. From Theorem \ref{MN} we know
\begin{align*}
    Q_k^*S_{\la}(\phi)=\sum_{\tau\subset\underline{\la}}\prod_{i=1}^{l(\la)}A(q^{|\phi|};\tau_i)(-1)^{|\tau||\phi|}S_{\la-\tau}Q_{k-|\tau||\phi|}^*+\cdots
\end{align*}
Note that $S_m(\phi)S_n(\psi)=S_n(\psi)S_m(\phi)$, so $Q^*_{k-(\la_1+l)|\phi|}S_{-l}(\phi)\ldots S_{\la_l}(\phi)S_{\tbl(\psi)}.1=0$, and similarly for other omitted terms. Then we conclude that for any two orbits $\phi,\psi\in\Phi^*$
\begin{align*}
Q^*_kS_{\tbl(\phi)}S_{\tbl(\psi)}.1=\sum_{\tau\subset\underline{\tbl(\phi)}}\prod_{i=1}^{l(\tbl(\phi))}A(q^{|\phi|};\tau_i)(-1)^{|\tau||\phi|}S_{\tbl(\phi)-\tau}(\phi)Q_{k-|\tau||\phi|}^*S_{\tbl(\psi)}.1
\end{align*}
Summarizing these we
have the following corollary immediately.

\begin{cor}\label{MMN}
    For a $\tbl\in\mathcal P^{\Phi^*}$ and a positive integer $k$, we have
\begin{align}\label{e:MN}
Q^*_k\prod_{\phi\in\Phi^*}S_{\tbl(\phi)}.1=\sum_{\tbtau\subset\underline{\tbl}}\mathcal A(q;\tbtau)(-1)^{\|\tbtau\|}\prod_{\phi\in\Phi^*}S_{\tbl(\phi)-\tbtau(\phi)}Q^*_{k-\|\tbtau\|}.1
\end{align}
summed over all $\tbtau\in\mathcal C^{\Phi^*}$ such that $\tbtau(\phi)_i>0, 1\leq i\leq l(\tbtau(\phi))$ for each $\phi\in\Phi^*$, and $\underline{\tbl}(\phi)$ is the  partition $\underline{\tbl(\phi)}$, and $\mathcal A(q;\tbtau)=\prod_{\phi\in\Phi^*}\prod_{i=1}^{l(\tbtau(\phi))}A(q^{|\phi|};\tbtau(\phi)_i)$.
\end{cor}
Using \eqref{e:MN}, we get an iterative formula to compute the character $\chi^{\tilde{\bl}}_{\mu(f_1)}$.

\begin{prop}\label{single1}
(1) Let $\mu=(\mu_1,\ldots,\mu_r)\vdash n$ be a partition and $m$ be a positive integer such that $n=m|\phi|$ for some $\phi\in\Phi^*$, the irreducible character of $\chi^{(m)(\phi)}_{\mu(f_1)}$ is the coefficient of $z^mw_1^{\,u_1}\ldots w_r^{\mu_r}$ of the formal series
\begin{align}\label{series}
    q^{n(\mu)}\prod_{1\leq i<j\leq r}\dfrac{w_i-w_j}{w_i-q^{-1}w_j}\prod_{i=1}^r[1-(-w_i)^{|\phi|}z]
\end{align}

(2) Let $\la=(\la_1,\ldots,\la_l)\vdash m$ be a partition and $n$ be a positive integer such that $n=m|\phi|$ for some $\phi\in\Phi^*$, the irreducible character of $\chi^{\la(\phi)}_{(n)(f_1)}$ is the coefficient of $z^nw_1^{\la_1}\ldots w_l^{\la_l}$ of the formal series
    \begin{align}
        \prod_{1\leq i<j\leq l}\left(1-\dfrac{w_j}{w_i}\right)\prod_{i=1}^l[1-(-z)^{|\phi|}w_i].
    \end{align}
In particular if $|\phi|=1$, we have
\begin{align}
    \chi^{\la(\phi)}_{(n)(f_1)}=\begin{cases}
        (-1)^{m+n}, &\mbox{if $\la=(1^m)$}\\
        0. &\mbox{otherwise}
    \end{cases}
\end{align}
\end{prop}
\begin{proof}
We only show the first statement, the second is similar. By Proposition \ref{ful:char} we know
\begin{align*}
    \chi^{(m)(\phi)}_{\mu(f_1)}=q^{n(\mu)}\langle S_m(\phi).1,Q_{\mu_1}\ldots Q_{\mu_r}.1\rangle.
\end{align*}
This inner product $\langle S_m(\phi).1,Q_{\mu_1}\ldots Q_{\mu_r}.1\rangle$ is the coefficient of $z^mw_1^{\mu_1}\ldots w_r^{\mu_r}$ in the following formal expression
\begin{align*}
 &\langle S(\phi;z).1,Q(f_1;w_r)\ldots Q(f_1;w_l).1 \rangle \\
 =&\prod_{1\leq i<j\leq r}\dfrac{w_i-w_j}{w_i-q^{-1}w_j}\langle S(\phi;z).1,:Q(f_1;w_1)\ldots Q(f_1;w_r):.1 \rangle \quad({\text{by relation \cite[Prop.2.9]{Jing2}}}) \\
=&\prod_{1\leq i<j\leq r}\dfrac{w_i-w_j}{w_i-q^{-1}w_j}\left\langle \exp\left(\sum_{n\geq1}\frac{1}{n}p_n(\phi)z^n\right),\prod_{i=1}^r\exp\left(\sum_{n\geq1}\dfrac{q^n-1}{n}p_n(f_1)w^n_i \right)\right\rangle \\
=&\prod_{1\leq i<j\leq r}\dfrac{w_i-w_j}{w_i-q^{-1}w_j}\prod_{i=1}^r[1-(-w_i)^{|\phi|}z]
\end{align*}
where the last equality is due to vertex operators.
\end{proof}

We know the formal series
    \begin{align*}
    \dfrac{w_i-w_j}{w_i-q^{-1}w_j}=\left(1-\dfrac{w_j}{w_i}\right)\sum_{k\geq0}^\infty \dfrac{w_j^k}{q^kw_i^k}=1+\sum_{k\geq1}(q^{-k}-q^{-k+1})\dfrac{w_j^k}{w_i^k}
    \end{align*}
Let us discuss the character value $\chi^{(m)(\phi)}_{\mu(f_1)}$ in several cases:

{\textbf{Case 1:}} If $l(\mu)=r<m$, it is clear that $\chi^{(m)(\phi)}_{\mu(f_1)}=0$.

{\textbf{Case 2:}} If $l(\mu)=r=m$, the coefficient of $z^m$ of $\prod_{i=1}^m[1-(-w_i)^{|\phi|}z]$ is $(-1)^{(|\phi|+1)m}w_1^{|\phi|}\ldots w^{|\phi|}_m$. Notice that the action of $\dfrac{w_j}{w_i},i<j$ will increase the power of $w_j$ and decrease the power of $w_i$. But $\mu_i\geq\mu_j$, it forces that $\mu_i=\mu_j$ for all $i,j=1,\ldots,m$ unless $\chi^{(m)(\phi)}_{\mu(f_1)}=0$.
\begin{align}
    \chi^{(m)(\phi)}_{\mu(f_1)}=\begin{cases}
        (-1)^{n+m}q^{n(m-1)/2}, &\mbox{if $\mu=(|\phi|^m)$}\\
        0. &\mbox{otherwise}
    \end{cases}
\end{align}

{\textbf{Case 3:}} If $l(\mu)=r>m$, we denote by $W$ the set $\{w_1^{\phi},\ldots,w_r^{\phi}\}$ and by $W_t$ the set $\{w_1^{|\phi|},\ldots,w_r^{|\phi|}\}$ except $w_t^{|\phi|}, 1\leq t\leq r$. We know that $\prod\limits_{i=1}^r[1-(-w_i)^{|\phi|}z]=\sum\limits_{i=1}^r(-1)^{({|\phi|}+1)i}e_i(W)z^i$, where $e_i$ is the elementary symmetric function in variables $w_1^{|\phi|},\ldots,w_r^{|\phi|}$. Denote by $D(m;W)$ the product $e_m(W)\prod\limits_{1\leq i<j\leq r}\dfrac{w_i-w_j}{w_i-q^{-1}w_j}$.

For any sequence $\alpha=(\alpha_1,\ldots,\alpha_r)$ of integers, we denote by $C^m_{\alpha}(q)$ the coefficient of $w_1^{\alpha_1}\ldots w_s^{\alpha_s}$ of $D(m;W)$.
We consider the polynomial $C_{\alpha}^m(q),m\geq1$.
\begin{align*}
&D(m;W)=e_m(w_1^{|\phi|},\ldots,w_r^{|\phi|})\prod_{1\leq i<j\leq r}\dfrac{w_i-w_j}{w_i-q^{-1}w_j}\\
=&\sum_{t=1}^rw_t^{|\phi|}e_{m-1}(W_t)\prod_{\substack{1\leq i<j\leq r\\i,j\neq t}}\dfrac{w_i-w_j}{w_i-q^{-1}w_j}\prod_{j=1}^{t-1}\dfrac{w_j-w_t}{w_j-q^{-1}w_t}\prod_{j=t+1}^r\dfrac{w_t-w_j}{w_t-q^{-1}w_j} \\
=&\sum_{t=1}^r\sum_{k_{t1},\ldots,k_{tr}}q^{-\sum\limits_{j=1}^rk_{tj}}(1-q)^{\sum\limits_{j=1}^r\delta_{k_{tj}>0}}D(m-1;W_t) w_t^{|\phi|+\sum\limits_{j=1}^{t-1}k_{tj}-\sum\limits_{j=t+1}^rk_{tj}}w_1^{-k_{ti}}\ldots w_{t-1}^{-k_{t,t-1}}w_{t+1}^{k_{t,t+1}}\ldots w_r^{k_{tr}}
\end{align*}
where $k_{tt}=0$ and $\delta_{k_{tj}>0}=\begin{cases}
    1,&\mbox{if $k_{tj}>0$}\\
    0, &\mbox{if $k_{tj}=0$}
\end{cases}$.

For a sequence $\alpha=(\alpha_1,\dots,\alpha_r)$, we define two new sequences $\alpha^{(i)},\alpha^{\{i\}}$ associated with $\alpha$ as follows
\begin{align*}
    \alpha^{(i)}=(\alpha_1,\ldots,\alpha_{(i-1)},\alpha_{(i+1)},\ldots,\alpha_r),\quad
  \alpha^{\{i\}}=(\alpha_1,\ldots,\alpha_{(i-1)},-\alpha_{(i+1)},\ldots,-\alpha_r)
\end{align*}
In particular $\alpha^{(1)}=\alpha^{[1]}$, $\alpha^{\{1\}}=-\alpha^{[1]}$ and $\alpha^{(r)}=\alpha^{\{r\}}$.
Since the coefficient of $w_1^{\alpha_1}\ldots w_{t-1}^{\alpha_{t-1}}w_{t+1}^{\alpha_{t+1}}\ldots w_r^{\alpha_r}$ of $D(m-1;W_t)$ is $C_{\alpha^{(t)}}^{m-1}(q)$, we have
\begin{align}
    C_{\alpha}^m(q)=\sum_{t=1}^r\sum_{\gamma}q^{-|\gamma|}(1-q)^{l(\gamma)}C_{\alpha^{(t)}+\gamma^{\{t\}}}^{m-1}
\end{align}
summed over all $\gamma\in\mathcal C$ such that $\gamma_t=0$, $|\gamma^{\{t\}}|=|\phi|-\alpha_t$ and $l(\gamma)\leq r$.  Similarly, for $m=0$
\begin{align}
    C_{\alpha}^0(q)=\sum_{\beta}q^{-|\beta|}(1-q)^{l(\beta)}C_{\alpha^{[1]}-\beta}^0(q)
\end{align}
summed over all compositions $\beta\vDash-\alpha_1$. For instance $\alpha=(\alpha_1,\alpha_2)$, it is easy to get
$C^0_{(\alpha_1, \alpha_2)}(q)=0$ for $\alpha_1+\alpha_2\neq 0$ or $\alpha_2<0$ and
$C^1_{(\alpha_1, \alpha_2)}(q)=0$ for $\alpha_1+\alpha_2\neq d$ or $\alpha_2<0$. The other nontrivial values are
\begin{align*}
    C_{(\alpha_1,\alpha_2)}^0(q)=\begin{cases}
    1, &\mbox{if $\alpha=(0,0)$}\\
    q^{-\alpha_2}(1-q). &\mbox{if $\alpha_1<0,\alpha_2>0$}
    \end{cases},\quad
    C_{(\alpha_1,\alpha_2)}^1(q)=\begin{cases}
    1, &\mbox{if $\alpha=(d,0)$ or $(0,d)$}\\
    q^{-\alpha_2}(1-q), &\mbox{if $\alpha_1>0,\alpha_2>0$}\\
    q^{\alpha_1}(1-q). &\mbox{if $\alpha_1<0,\alpha_2>0$}
    \end{cases}
\end{align*}

For a partition $\mu$, $C_{\mu}^m(q)$ is the coefficient of $w_1^{\mu_1}\ldots w_r^{\mu_r}$ in $D(m;W)$. Notice that $\mu_1\geq\ldots\geq\mu_r$, the only contributing term
in the expansion of $D(m;w)$ is $w_1^d\prod\limits_{2\leq i<j\leq r}\dfrac{w_i-w_j}{w_i-q^{-1}w_j}D(m-1;W_1)$, thus,
\begin{align}
    C_{\mu}^m(q)=\sum_{\gamma}q^{-|\gamma|}(1-q)^{l(\gamma)}C_{\mu^{(1)}-\gamma}^{m-1}
\end{align}
summed over all $\gamma\vDash |\phi|-\mu_1$. It is clear that $C_{\mu}^m(q)=0$ if $\mu_1>d$. The corresponding character value
\begin{align}
    \chi^{(m)(\phi)}_{\mu(f_1)}=(-1)^{n+m}q^{n(\mu)+\mu_1-|\phi|}\sum_{\gamma}(1-q)^{l(\gamma)}C_{\mu^{(1)}-\gamma}^{m-1}
\end{align}

If $|\phi|=1$, we know $\chi_{\mu(f_1)}^{(n)(\phi)}$ is the coefficient of $z^nw^{\mu}$ of the formal series \eqref{series}, thus
\begin{align*}
    \chi_{\mu(f_1)}^{(n)(\phi)}=\begin{cases}
     q^{n(n-1)/2}, &\mbox{if $\mu=(1^n)$}  \\
      0. &\mbox{if $\mu\neq(1^n)$}
    \end{cases}
\end{align*}

Let $\phi_1$ be the $F$-orbit of the trivial character $\hat{1}\in M_1^*$. For each $f\in\Phi$ We have the bracket product
\begin{align*}
    \left[\exp\left(\sum_{n\geq1}\dfrac{1}{n}p_n(\phi_1)z^n\right),\exp\left(\sum_{n\geq1}\dfrac{q^{n|f|}-1}{n}p_n(f)w_i^n\right)\right]=1-(-z)^{|f|}w.
\end{align*}
From Proposition \ref{single1}, for each $\bmu\in\mathcal P_n^{\Phi}$, we similarly deduce that $\chi^{(n)(\phi_1)}_{\bmu}$ is the coefficient $z^nw^{\bmu}$ of the following formal series
\begin{align*}
    q^{n(\bmu)}\prod_{1\leq i<j\leq l(\bmu(f))}\dfrac{w_i-w_j}{w_i-q^{-|f|}w_j}\prod_{i=1}^{l(\bmu(f))}[1-(-z)^{|f|}w_i].
\end{align*}
Then the following proposition is immediate.

\begin{prop}\cite[Theorem 9.2]{CLT}\cite[1.13]{SZ}\label{St}
    Let $\bmu\in\mathcal P_n^{\Phi}$, the irreducible character $\chi^{(n)(\phi_1)}_{\bmu}$ is given by
    \begin{align}
        \chi^{(n)(\phi_1)}_{\bmu}=\begin{cases}
            (-1)^{n-\sum_f l(\bmu(f))}q^{n(\bmu)},&\mbox{$\bmu(f)=(1^{l(\bmu(f))})$}\\
            0. &\mbox{otherwise}
        \end{cases}
    \end{align}
\end{prop}
\begin{rem}
    Recall that an element $x\in\GL_n(\f_q)$ is semisimple if $x$ is of type $\bmu(f)=(1^{l(\bmu(f))})$ for each $f\in\Phi$. This irreducible character $\chi^{(n)(\phi_1)}$ is the famous Steinberg character.
\end{rem}

\begin{cor}
 (1)    For each $\Phi^*$-colored partition $\tbl\in\mathcal P_n^{\Phi^*}$ such that $l(\tbl(\phi))=1$ for each $\phi$, let $\mu=(\mu_1,\ldots,\mu_r)\vdash n$ be a partition, the irreducible character $\chi^{\tbl}_{\mu(f_1)}$ is the coefficient of $z^{\tbl}w^{\mu}$ of the formal series
     \begin{align}
         q^{n(\mu)}\prod_{1\leq i<j\leq r}\dfrac{w_i-w_j}{w_i-q^{-1}w_j}\prod_{\phi\in\Phi^*}[1-(-w_i)^{|\phi|}z]
     \end{align}
     Moreover, if all $|\phi|=1$ for $\tbl(\phi)\neq0$, then we have
     \begin{align*}
        \chi_{\mu(f_1)}^{\tbl}=\begin{cases}
            q^{n(n-1)/2}, &\mbox{if $\mu=(1^n)$}  \\
            0. &\mbox{if $\mu\neq(1^n)$}
        \end{cases}
    \end{align*}

(2) For each $\tbl\in\mathcal P_n^{\Phi^*}$, the irreducible character $\chi^{\tbl}_{(n)(f_1)}$ is given by
\begin{align}
    \chi^{\tbl}_{(n)(f_1)}=\begin{cases}
        (-1)^{n+\sum_{\phi}l(\tbl(\phi))},&\mbox{if $\tbl(\phi)=(1^{l(\tbl(\phi))})$ for each $\phi$ }\\
        0.&\mbox{otherwise}
    \end{cases}
\end{align}
\end{cor}

We now give the Murnaghan-Nakayama rule for unipotent classes, expresses the irreducible character of $\GL_n(\f_q)$ in term of
the general linear groups of lower dimension. For each $\tbtau\in\mathcal C^{\Phi^*}$ we define the polynomial
\begin{align}
    \mathcal{B}(q;\tbtau)=\prod_{\phi\in\Phi^*}\prod_{i=1}^{l(\tbtau(\phi))}B(q^{|\phi|};\tbtau(\phi)_i),\quad B(q^{|\phi|};\tbtau(\phi)_i)=\begin{cases}
    -1,&\mbox{if $\tbtau(\phi)_i=1$}\\
    q^{-1}-1.&\mbox{if $\tbtau(\phi)_i>1$}
\end{cases}
\end{align}
\begin{thm}\label{table}
    For each $\Phi^*$-colored partition $\tbl\in\mathcal P_n^{\Phi^*}$, let $\mu=(\mu_1,\ldots,\mu_r)\vdash n$ be a partition, the irreducible character $\chi^{\tbl}_{\mu(f_1)}$ is given by
    \begin{align}\label{cha table}
\chi^{\tbl}_{\mu(f_1)}=(-1)^{\mu_1}q^{|\mu|}\sum_{\tbtau\subset\underline{\tbl}}\mathcal B(q;\tbtau)\sum_{\trh}\chi^{\trh'}_{\mu^{[1]}(f_1)}.
    \end{align}
summed over all $\tbtau\in\mathcal C^{\Phi^*}$ such that $\tbtau(\phi)_i>0, 1\leq i\leq l(\tbtau(\phi))$ and $\trh\in\mathcal P^{\Phi^*}_{(\|\tbtau\|-\mu_1)}$ with $l(\trh(\phi))=1$ for each $\phi\in\Phi^*$, the corresponding $\trh'$ is defined in \eqref{trh-trh} associated with $\trh$.
\end{thm}
\begin{proof}
    It follows from Theorem \ref{thm:q-S} and Corollary \ref{MMN} that
    \begin{align*}
    &\langle S_{\tbl}.1, Q_{\mu}.1\rangle=\left\langle Q_{\mu_1}^*\prod_{\phi\in\Phi^*}S_{\tbl(\phi)}.1,Q_{\mu^{[1]}}.1 \right\rangle\\
   =&\left\langle\sum_{\tbtau\subset\underline{\tbl}}\mathcal A(q;\tbtau)(-1)^{\|\tbtau\|}\prod_{\phi\in\Phi^*}S_{\tbl(\phi)-\tbtau(\phi)}(-\dfrac{1}{q})^{\|\tbtau\|-\mu_1}\sum_{\trh}S_{\trh(\phi)}.1, Q_{\mu^{[1]}}.1 \right\rangle
  \end{align*}
where $\|\trh\|=\|\tbtau\|-\mu_1$, $l(\trh(\phi))=1$ for each $\phi$. We remark that if $\|\tbtau\|-\mu_1<0$, $Q^*_{\mu_1-\|\tbtau\|}.1=0$. For each $\phi\in\Phi^*$,
$\tbl(\phi)-\tbtau(\phi)$ can always be normalized into a partition by using since \eqref{rel:Schur} up to a sign. We define a $\Phi^*$-colored composition $\trh'$ associated with $\trh$ such that
\begin{align}\label{trh-trh}
    \trh'(\phi)=[\tbl(\phi)-\tbtau(\phi),\trh(\phi)].
\end{align}
The above inner product can be written as
\begin{align*}
&\left\langle\sum_{\tbtau\subset\underline{\tbl}}\mathcal A(q;\tbtau)q^{\mu_1-\|\tbtau\|}(-1)^{\mu_1}\sum_{\trh}\prod_{\phi\in\Phi^*}S_{\trh'(\phi)}, Q_{\mu^{[1]}}.1  \right\rangle=\left\langle\sum_{\tbtau\subset\underline{\tbl}}\mathcal B(q;\tbtau)q^{\mu_1}(-1)^{\mu_1}\sum_{\trh}\prod_{\phi\in\Phi^*}S_{\trh'(\phi)}, Q_{\mu^{[1]}}.1  \right\rangle
\end{align*}
Notice that $n(\mu(f_1))-n(\mu^{[1]}(f_1))=|\mu^{[1]}|$, then the formula \eqref{cha table} is proved.
\end{proof}

\begin{rem}
It should be pointed out that we don't consider $\tbl(\phi)-\tbtau(\phi)$ as a sequence rather than a partition. Before using Theorem \ref{MMN}, the readers need to arrange the order of $S_{\tbl(\phi)-\tbtau(\phi)}.1$, even the order of $S_{\tbl(\phi)-\tbtau(\phi)}S_{\trh(\phi)}.1$ themselves.
\end{rem}

Recall that a partition is hook-shaped of $n$ if it is
of the form $(k,1^{n-k}),1\leq k\leq n$.

\begin{prop}\label{hook}
    For each $\Phi^*$-colored partition $\tbl\in\mathcal P_n^{\Phi^*}$, the irreducible character $\chi^{\tbl}_{(k,1^{n-k})(f_1)}$ is given by
    \begin{align}
        \chi^{\tbl}_{(k,1^{n-k})(f_1)}=(-1)^{\mu_1}q^{|\mu|}\sum_{\tbtau\subset\underline{\tbl}}\mathcal B(q;\tbtau)\sum_{\trh}d(\trh')
    \end{align}
summed over all $\tbtau\in\mathcal C^{\Phi^*}$ such that $\tbtau(\phi)_i>0, 1\leq i\leq l(\tbtau(\phi))$ and $\trh\in\mathcal P^{\Phi^*}_{(\|\tbtau\|-\mu_1)}$ with $l(\trh(\phi))=1$ for each $\phi\in\Phi^*$, the corresponding $\trh'$ is defined in \eqref{trh-trh} associated with $\trh$. And $d(\trh')$ is the degree of $\trh'$ as in \eqref{degree}.
\end{prop}
\begin{proof}
By Green's work \cite{jaG}(cf. \cite[pp. 284-286]{Mac} and our approach \cite[Sec.5]{JW}), the degree
of the irreducible character is given by
    \begin{align}\label{degree}
        d(\tbl)=\prod_{i=1}^n(q^i-1)\prod_{\phi\in\Phi^*}\dfrac{q_\phi^{n(\tbl(\phi)')}}{\prod_{x\in\tbl(\phi)}(q_\phi^{h(x)}-1)}
    \end{align}
    where $h(x)=\la_i+\la_j'-i-j+1$ is the hook length of the node $x=(i,j)$ inside the Young diagram of shape $\tbl(\phi)$.
    The proposition is then followed by Theorem \ref{table}.
\end{proof}

\bigskip
\bibliographystyle{plain}

\end{document}